\RequirePackage{fix-cm}
\documentclass[smallextended]{svjour3}       
\smartqed  
\usepackage{graphicx}
\usepackage{amsmath}
\usepackage{amsfonts}
\begin{document}

\title{ Inverse approximation and GBS of bivariate Kantorovich type sampling series 
}


\author{A. Sathish Kumar         \and
         Bajpeyi Shivam $^{*}$ 
}



\institute{A. Sathish Kumar  \at
              Department of Mathematics, Visvesvaraya National Institute of Technology, Nagpur-440010
               \\
                          \email{mathsathish9@gmail.com } \vspace{4mm} \at  Bajpeyi Shivam$^{*}$(corresponding author)\at Department of Mathematics, Visvesvaraya National Institute of Technology, Nagpur-440010 \at \email{ shivambajpai1010@gmail.com}  }

\date{Received: date / Accepted: date}

\maketitle

\begin{abstract}

In this paper, we derive an inverse result for bivariate Kantorovich type sampling series for $ f \in C^{2}(\mathbb{R}^{2})$ (the space of all continuous functions with upto second order partial derivatives are continuous and  bounded on $ \mathbb{R}^{2}).$ Further, we prove the rate of approximation in the B$\ddot{o}$gel space of continuous functions for the GBS (Generalized Boolean Sum) of these operators. Finally, we give some examples for the kernel to which the theory can be applied.

\keywords{Inverse result \and Bivariate Kantorovich type Sampling series \and Rate of convergence \and GBS operators  \and mixed modulus of smoothness }
\subclass{41A25 \and 94A20 \and 41A30 \and 41A10}
\end{abstract}

\section{Introduction}
\label{intro}

The theory of generalized sampling series was first initiated by P. L. Butzer and his school \cite{Stens1} and \cite{Stens2}. Let $\chi :{\mathbb{R}}^2\rightarrow {\mathbb{R}}$ be a suitable kernel function.  Then, the two dimensional generalized sampling series of a function  $f:{\mathbb{R}}^2\rightarrow {\mathbb{R}}$ is defined by
\begin{eqnarray}\label{e1}
(G_w^{\chi}f)(x,y)&=&\sum_{k=-\infty}^{\infty}\sum_{j=-\infty}^{\infty}\chi \big (wx-k, wy-j \big) f \bigg(\frac{k}{w},\frac{j}{w}\bigg)
\end{eqnarray}
where $w > 0 $ and $(x,y)\in {\mathbb{R}}^2.$ These operators have great importance in the development of models for signal recovery. These type of operators have been studied by many authors (cf. \cite{bv1,BM*,PLB2,PLB1} etc.).\\
The Kantorovich type generalizations of approximation operators is an important subject in
approximation theory and they are the method to approximate Lebesgue integrable functions.
In the last few decades, the Kantorovich modifications of several operators were constructed and their approximation behaviour have been studied. We mention some of the work in this direction e.g., \cite{GV3,PNAI,gupta1,BM1,Tuncer,gupta2} etc.

Bardaro et.al.\cite{BVBS}, studied the rate of convergence of sampling Kantorovich operators in the general settings of Orlicz spaces. Danilo and Vinti \cite{COS}, Bardaro \cite{BM} extended their study in the multivariate setting and obtained the rate of convergence for functions in Orlicz spaces and also obtained the rate of approximation for the family of sampling Kantorovich operators in the uniform norm, for bounded and uniformly continuous, functions belonging to Lipschitz classes and for functions in Orlicz spaces in \cite{Dani}. The non-linear version of sampling Kantorovich series has been studied in \cite{Dan}, \cite{Zam} and \cite{DV}.
Butzer and Stens studied the saturation results for the univariate generalised sampling operators in (\cite{Stens2}). Recently, the inverse results for the Kantorovich type sampling operators have been derived for the spaces $C^{2}(\mathbb{R})$ and $C^{1}(\mathbb{R})$ in (\cite{GV3**}), (\cite{BCG}) respectively.\\

Let $ C^{r}(\mathbb{R} ^{2})$ denotes the space of all continuous functions with upto $r$-th order partial derivatives are continuous and
 bounded on $ \mathbb{R}^{2}.$ Assume that $\chi\in C^{0}( \mathbb{R}^{2})$ be fixed. For any $ \eta \in \mathbb{N}_{0}= \mathbb{N} \cup \{0 \},$ $ p= (p_{1},p_{2})\in \mathbb{N}_{0}^2$ with $ \vert p \vert = p_{1}+p_{2}= \eta,$ we define the algebraic moments as,
\begin{eqnarray*}
m_{( p_{1},p_{2})}^{\eta} (\chi,u,v):=\sum_{k=-\infty}^{\infty}\sum_{j=-\infty}^{\infty}
\chi(u-k,v-j)(u-k)^{p_{1}}(v-j)^{p_{2}}
\end{eqnarray*}
and the absolute moments by,
\begin{eqnarray*}
M_{( p_{1},p_{2})}^{\eta} (\chi,u,v):= \sup_{(x,y)\in\mathbb{R}^2}\sum_{k=-\infty}^{\infty}\sum_{j=-\infty}^{\infty}|\chi(u-k,v-j)|\, |u-k|^{p_{1}}\, |y-j|^{p_{2}}
\end{eqnarray*}
and,
$$ M_{\eta}(\chi):= \max_{\vert p \vert = \eta} \  M_{(p_{1},p_{2})}^{\eta} (\chi). $$
We can easily see that, for $ \xi,\eta \in \mathbb{N}_{0}$ with $\xi < \eta,$  $ M_{\eta}(\chi)< + \infty$ implies that $M_{\xi}(\chi)< + \infty.$
Also note that, when $\chi$ is compactly supported then we have, $ M_{\eta}(\chi)< \infty$, for every $ \eta \in \mathbb{N}_{0}.$ (see \cite{bv1})\\

The bivariate Kantorovich version of the generalized sampling series is defined as,
\begin{equation} \label{e2}
(S_{w} f)(x,y) = \sum_{k \in Z} \sum_{j \in Z} \chi(wx-k, wy-j)
\Bigg[ w^{2}  \int_{\frac{k}{w}}^{\frac{k+1}{w}} \int_{\frac{j}{w}}^{\frac{j+1}{w}} f(u,v)\  du \ dv \Bigg ], \ \ \ \ \  \forall (x,y) \in \mathbb{R}^{2},
\end{equation}

where $f : \mathbb{R}^{2}\rightarrow \mathbb{R} $ is a locally integrable function, such that the above series is convergent for every $(x,y) \in \mathbb{R}^{2}$ and $ \chi: \mathbb{R}^{2} \rightarrow \mathbb{R}$ is a kernel function.\\

We assume that the following conditions hold;\\

$(\textbf{K1}) \ \ \chi $ is integrable and bounded at the origin.\\

$(\textbf{K2}) \ $ The series $ \displaystyle{\sum_{k \in \mathbb{Z}} \sum_{j \in \mathbb{Z}}} \chi(u-k, v-j) =1 ,$ for every $(u,v) \in \mathbb{R} ^{2},$ \\

$(\textbf{K3})\ \ M_{(p_{1},p_{2})}^{\eta}(\chi) < +\infty $ , for some $ (p_{1},p_{2}) \in \mathbb{N}^{2}.$ \\

\begin{theorem}
Let $\chi$ be the kernel satisfying the following moment condition,
\begin{equation} \label{e3}
 m_{(p_{1},p_{2})}^{ \eta}(\chi,u,v) =
     \begin{cases}
      \text{0,} &\quad\text{if,} \ \ \ \ {\eta =1,2,3...,r-1} \\
     \text{c,} &\quad\text{if,} \ \ \ \ { \eta =r}\\
   \end{cases}
\end{equation}
for every $(u,v) \in{\mathbb{R}}^{2}$ and some $ r \in \mathbb{N}^{+}.$ Then, for every $ f \in C^{r}(\mathbb{R}^{2})$, we have,
$$ \Vert G_{w}f - f \Vert _{\infty} \leq \frac{c}{r!}\frac{M_{r} (\chi)}{ w^{r}}. H $$ where, $$ M_{r} (\chi)=\max _{\eta = r} \  M_{(p_{1},p_{2})}^{r} (\chi).$$
Here,
  $ \ H = \binom {r}{0} A_{r} + \binom {r}{1} A_{r-1} B_{1} + \binom {r}{2} A_{r-2} B_{2}+...+ \binom {r}{r} B_{r} ,$  with
$$ A_{r} = \Big \Vert \frac{\delta ^{r}f}{\delta x^{r}} \Big \Vert_{\infty}  \ and \ B_{r} = \Big \Vert \frac{\delta ^{r}f}{\delta y^{r}} \Big \Vert_{\infty} .$$
\end{theorem}

\begin{proof}{\textbf{.}} By using (\ref{e1}), we have
\begin{eqnarray*} \label{e4}
|(G_{w} f)(x,y)-f(x,y)|&=& \Big \vert  \sum_{k \in \mathbb{Z}} \sum_{j \in \mathbb{Z}} \chi(wx-k, wy-j) f \Big(\frac{k}{w}, \frac{j}{w} \Big) - f(x,y) \Big \vert\\
&=&  \Big \vert \sum_{k \in \mathbb{Z}} \sum_{j \in \mathbb{Z}} \chi(wx-k, wy-j) \Big( f \Big(\frac{k}{w}, \frac{j}{w} \Big) - f(x,y)\Big) \Big \vert.
\end{eqnarray*}
Let $ f \in C^{r} (\mathbb{R}^{2}).$ Now, by the Taylor's formula in two variables with Lagrange's form of remainder, we have
\begin{eqnarray*}
f(u,v)&=& f (x,y) +  \frac{\delta f}{\delta x}(x,y ) \ (u - x) +  \frac{\delta f}{\delta y} (x,y) \ (v-y) +...\\&&+ \frac{1}{r!} \Bigg[  \binom {r}{0} \frac{\delta ^{r}f}{\delta x^{r}} \big( \theta_{u,x},\theta_{v,y} \big) \ ( u - x)^{r}  + \binom {r}{1} \frac{\delta ^{r}f}{\delta x^{r-1} \delta y } \big( \theta_{u,x},\theta_{v,y} \big) \ ( u - x)^{r-1} ( v - y)  +
\\&&
\binom {r}{2}  \frac{\delta ^{r}f}{\delta x^{r-2} \delta y^{2}}\big( \theta_{u,x},\theta_{v,y} \big)\ ( u - x)^{r-2} ( v - y)^{2}  +...
+ \binom {r}{r} \frac{\delta ^{r}f}{\delta y^{r}} \big( \theta_{u,x},\theta_{v,y} \big) \ ( v - y)^{r}    \     \Bigg ],
\end{eqnarray*}
$\mbox{where},(x,y) ,(u,v) \in \mathbb{R}^{2} \ \ and \  \theta_{u,x} \in (x,u) \ \mbox{and} \  \theta_{v,y} \in (y,v).$
Now, if we substitute  $ x= \frac{k}{w} \ \mbox{and} \  y= \frac{j}{w}, \ \ \  k,j \in \mathbb{Z}, \ w > 0,$ for every $ u \in \Big( \frac{k}{w},\frac{k+1}{w} \Big) , \\ v \in \Big(\frac{j}{w},\frac{j+1}{w} \Big)$ in the above formula, then,
\begin{eqnarray*}
f(u,v)&=& f \Big(\frac{k}{w}, \frac{j}{w} \Big) +  \frac{\delta f}{\delta x} \Big(\frac{k}{w}, \frac{j}{w}\Big ) \Big( u - \frac{k}{w} \Big) +   \frac{\delta f}{\delta y} \Big(\frac{k}{w}, \frac{j}{w} \Big) \Big( v - \frac{j}{w} \Big)+...\\&&+ \frac{1}{r!} \Bigg[  \binom {r}{0} \frac{\delta ^{r}f}{\delta x^{r}} \big( \theta_{u,\frac{k}{w}},\theta_{v,\frac{j}{w}} \big) \Big(u - \frac{k}{w} \Big)^{r}  + \binom {r}{1} \frac{\delta ^{r}f}{\delta x^{r-1} \delta y } \big( \theta_{u,\frac{k}{w}},\theta_{v,\frac{j}{w}} \big) \Big( u - \frac{k}{w} \Big)^{r-1} \Big( v - \frac{j}{w}\Big)  +
\\&&
\binom {r}{2}  \frac{\delta ^{r}f}{\delta x^{r-2} \delta y^{2}}\big( \theta_{u,\frac{k}{w}},\theta_{v,\frac{j}{w}} \big) \Big( u - \frac{k}{w} \Big)^{r-2}\Big( v - \frac{j}{w} \Big)^{2}  +...
+ \binom {r}{r} \frac{\delta ^{r}f}{\delta y^{r}}\big( \theta_{u,\frac{k}{w}},\theta_{v,\frac{j}{w}} \big) \Big( v - \frac{j}{w} \Big)^{r}    \     \Bigg ].
\end{eqnarray*}

Now, we have \\

\noindent
$|(G_{w} f)(x,y)-f(x,y)|$
\begin{eqnarray*}
&=&  \Bigg \vert \sum_{k \in Z} \sum_{j \in Z}  \chi(wx-k, wy-j)\Bigg [  \frac{\delta f}{\delta x} \Big(\frac{k}{w}, \frac{j}{w} \Big ) \Big( x - \frac{k}{w}\Big) +  \frac{\delta f}{\delta y} \Big(\frac{k}{w}, \frac{j}{w} \Big) \Big( y - \frac{j}{w}\Big)+...\\&& + \frac{1}{r!} \Bigg (  \binom {r}{0} \frac{\delta ^{r}f}{\delta x^{r}}\big( \theta_{u,\frac{k}{w}},\theta_{v,\frac{k}{w}} \big)  \Big( x - \frac{k}{w}\Big)^{r}  + \binom {r}{1} \frac{\delta ^{r}f}{\delta x^{r-1} \delta y } \big( \theta_{u,\frac{k}{w}},\theta_{v,\frac{k}{w}} \big) \Big( x - \frac{k}{w}\Big)^{r-1}\Big( y - \frac{j}{w}\Big)  +  \\&&
\binom {r}{2}  \frac{\delta ^{r}f}{\delta x^{r-2} \delta y^{2}} \big( \theta_{u,\frac{k}{w}},\theta_{v,\frac{j}{w}} \big) \Big( x - \frac{k}{w}\Big)^{r-2} \Big( y - \frac{j}{w}\Big)^{2}  +...+ \binom {r}{r} \frac{\delta ^{r}f}{\delta y^{r}}\big( \theta_{u,\frac{k}{w}},\theta_{v,\frac{j}{w}} \big) \Big( y - \frac{j}{w}\Big)^{r}  \Bigg) \Bigg ] \Bigg \vert.
\end{eqnarray*}
Now, using (\ref{e3}) and then taking the supremum norm on both sides, we obtain,
$$ \Vert G_{w}f - f \Vert _{\infty} \leq \frac{c}{r!}\frac{M_{r} (\chi)}{ w^{r}}. H $$ where, $$ M_{r} (\chi)=\max _{\eta = r} \  M_{(p_{1},p_{2})}^{r} (\chi).$$ \\
This completes the proof.
\end{proof}

\begin{theorem}
Let $\chi$ be kernel, which satisfies the moment condition (\ref{e3}) then,
$$(G_{w} P_{r-1})(x,y)= ( P_{r-1})(x,y)$$ for every $ w > 0$, i.e., $ (G_{w}f)$ maps a homogeneous polynomial of degree $(r-1)$ into a homogeneous polynomial of same degree, where $(P_{r-1})(x,y)$ denotes any homogeneous algebraic polynomial of degree $r-1.$
 \end{theorem}
\begin{proof}{.}
Proceeding in a manner similar to the proof of Remark 2 in  \cite{PLB2}, we can easily obtain the proof of this theorem.
\end{proof}

Now, we derive the representation formula which relates the bivariate Kantorovich sampling series with the bivariate generalized sampling operators.

\begin{theorem}
Let $ f \in C^{(2)}(\mathbb{R}^{2}).$ Then,
\begin{eqnarray*}
(S_{w} f)(x,y)= (G_{w} f)(x,y)+ \frac{1}{2w} (G_{w}f_{x})(x,y) + \frac{1}{2w} (G_{w}f_{y})(x,y) +\hat R_{2}^{w}(x,y) ,
\end{eqnarray*}
where, the remainder $\hat R_{2}^{w}(x,y)$ of order $2$, is the absolutely convergent series for every $(x,y) \in \mathbb{R}^{2}$ as;
\begin{eqnarray*}
{\hat R_{2}^{w}(x,y)= \frac{1}{2!} \sum_{k \in \mathbb{Z}} \sum_{j \in \mathbb{Z}} \chi(wx-k, wy-j)
\Bigg[ w^{2} \int_{\frac{k}{w}}^{\frac{k+1}{w}} \int_{\frac{j}{w}}^{\frac{j+1}{w}}  \Big \{  f_{xx} \big( \theta_{k,w}(u),\theta_{j,w}(v) \big)}
\bigg( u - \frac{k}{w}\bigg)^{2} \\+ 2  f_{xy} \big( \theta_{k,w}(u),\theta_{j,w}(v) \big)
\bigg(u - \frac{k}{w}\bigg) \bigg( v - \frac{j}{w}\bigg)  +
   f_{yy} \big( \theta_{k,w}(u),\theta_{j,w}(v) \big)
\bigg( v - \frac{j}{w}\bigg)^{2} \Big \} du dv \Bigg ]  ,
\end{eqnarray*}
where, $ \theta_{k,w}(u)$ and $\theta_{j,w}(v)$ are measurable functions, such that $ \frac{k}{w} <  \theta_{k,w}(u) < \frac{k+1}{w} $ and  $ \frac{j}{w} <  \theta_{j,w}(v) < \frac{j+1}{w}, \ \ k,j \in \mathbb{Z} ,$ and for every $u \in \Big( \frac{k}{w},\frac{k+1}{w} \Big) ,\\ v \in \Big(\frac{j}{w},\frac{j+1}{w} \Big), \ w > 0.$
\end{theorem}

\begin{proof}{.}
By the Taylor's formula with Lagrange's form of remainder for $ r=2$, we have,
\begin{eqnarray*}
f(u,v)&=& f (x,y) + f_{x} (x,y ) (u - x) +  f_{y}(x,y) (v-y) +
\frac{1}{2!} \Big[  f_{xx} \big( \theta_{u,x},\theta_{v,y} \big) \\&&
( u - x)^{2} + 2  f_{xy} \big( \theta_{u,x},\theta_{v,y} \big) ( u - x) ( v - y)  +
   f_{yy} \big( \theta_{u,x},\theta_{v,y} \big) ( v - y)^{2}  \Big ],
\end{eqnarray*}
where $(x,y) \ and \ (u,v) \in \mathbb{R}^{2} \ \ and \  \theta_{u,x} \in (x,u) \ \mbox{and} \  \theta_{v,y} \in (y,v).$
Now, if we substitute  $ x= \frac{k}{w} \ \mbox{and} \  y= \frac{j}{w}, \ \ \  k,j \in \mathbb{Z}, \ w > 0,$ for every $ u \in \Big( \frac{k}{w},\frac{k+1}{w} \Big) , v \in \Big(\frac{j}{w},\frac{j+1}{w} \Big)$ in the above formula, then $ \frac{k}{w} < \theta_{u,\frac{k}{w}} := \theta_{k,w}(u) < \frac{k+1}{w} $ and  $ \frac{j}{w} < \theta_{v,\frac{j}{w}} := \theta_{j,w}(v) < \frac{j+1}{w}.$ So, we have
\begin{eqnarray*}
f(u,v)&=& f \Big(\frac{k}{w},\frac{j}{w}\Big) + f_{x} \Big(\frac{k}{w},\frac{j}{w}\Big) \Big(u - \frac{k}{w}\Big) +  f_{y} \Big(\frac{k}{w},\frac{j}{w} \Big) \Big(v-\frac{j}{w}\Big) + \frac{1}{2!} \Bigg[   f_{xx} \big( \theta_{u,\frac{k}{w}},\theta_{v,\frac{j}{w}} \big) \\&&
\Big( u - \frac{k}{w}\Big)^{2}  + 2  f_{xy} \big( \theta_{u,\frac{k}{w}},\theta_{v,\frac{j}{w}} \big) \Big( u - \frac{k}{w}\Big) \Big( v - \frac{j}{w}\Big)  +
   f_{yy} \big( \theta_{u,\frac{k}{w}},\theta_{v,\frac{j}{w}} \big) \Big( v - \frac{j}{w}\Big)^{2}  \Bigg ].
 \end{eqnarray*}

\noindent
$ \mbox{Consider}, \displaystyle w^{2} \int_{\frac{k}{w}}^{\frac{k+1}{w}} \int_{\frac{j}{w}}^{\frac{j+1}{w}} f(u,v) \ du \ dv $
\begin{eqnarray} \label{e4}
\nonumber
&=& \displaystyle w^{2} \int_{\frac{k}{w}}^{\frac{k+1}{w}} \int_{\frac{j}{w}}^{\frac{j+1}{w}} \Bigg \{
 f \Big(\frac{k}{w}, \frac{j}{w} \Big) + f_{x} \Big(\frac{k}{w}, \frac{j}{w}\Big ).\Big( u - \frac{k}{w}\Big) +  f_{y} \Big(\frac{k}{w}, \frac{j}{w} \Big).\Big( v - \frac{j}{w}\Big) \\&& \nonumber
 + \frac{1}{2!} \Bigg[   f_{xx} \big( \theta_{k,w}(u),\theta_{j,w}(v) \big) \Big( u - \frac{k}{w} \Big)^{2}  + 2  f_{xy} \big( \theta_{k,w}(u),\theta_{j,w}(v) \big) \Big(u - \frac{k}{w}\Big) \Big( v - \frac{j}{w}\Big)  + \\&& \nonumber
  f_{yy} \big( \theta_{k,w}(u),\theta_{j,w}(v) \big) \Big( v - \frac{j}{w}\Big)^{2}  \Bigg ] \Bigg \} du dv\\
 &=&
f \Big(\frac{k}{w}, \frac{j}{w} \Big) + \frac{1}{2w} f_{x} \Big(\frac{k}{w}, \frac{j}{w}\Big ) + \frac{1}{2w} f_{y} \Big(\frac{k}{w}, \frac{j}{w}\Big )+ R_{2}^{w},
\end{eqnarray}
where, $ \displaystyle R_{2}^{w}=  w^{2} \int_{\frac{k}{w}}^{\frac{k+1}{w}} \int_{\frac{j}{w}}^{\frac{j+1}{w}} \frac{1}{2!} \Bigg[   f_{xx} \big( \theta_{k,w}(u),\theta_{j,w}(v) \big)
\bigg( u - \frac{k}{w}\bigg)^{2}  + 2  f_{xy} \big( \theta_{k,w}(u),\theta_{j,w}(v) \big) \\
\bigg(u - \frac{k}{w}\bigg) \bigg( v - \frac{j}{w}\bigg)  +
   f_{yy} \big( \theta_{k,w}(u),\theta_{j,w}(v) \big) \bigg( v - \frac{j}{w}\bigg)^{2}  \Bigg ] du dv .$ \\

Now, by using (\ref{e4}) in the defintion of $ (S_{w}f)(x,y)$, we have,
\begin{eqnarray*}
(S_{w} f)(x,y)
 &= &(G_{w} f)(x,y)+ \frac{1}{2w} (G_{w}f_{x})(x,y) + \frac{1}{2w} (G_{w}f_{y})(x,y)
  + \hat R_{2}^{w}(x,y),
  \end{eqnarray*}
where, $ \displaystyle \hat R_{2}^{w}(x,y)= \frac{1}{2!} \sum_{k \in \mathbb{Z}} \sum_{j \in \mathbb{Z}} \chi(wx-k, wy-j)\   \
\Bigg[ w^{2} \int_{\frac{k}{w}}^{\frac{k+1}{w}} \int_{\frac{j}{w}}^{\frac{j+1}{w}}  \Big \{  f_{xx} \big( \theta_{k,w}(u),\theta_{j,w}(v) \big)\\
\bigg( u - \frac{k}{w}\bigg)^{2}  + 2  f_{xy} \big( \theta_{k,w}(u),\theta_{j,w}(v) \big)
\bigg(u - \frac{k}{w}\bigg) \bigg( v - \frac{j}{w}\bigg) \\ +
   f_{yy} \big( \theta_{k,w}(u),\theta_{j,w}(v) \big) \bigg( v - \frac{j}{w}\bigg)^{2} \Big \} du dv \Bigg ]   .$ \\

Indeed, the above series $\hat R_{2}^{w}(x,y)$ is  absolutely convergent for every \\$ (x,y) \in \mathbb{R}^{2}$ and for every $ w > 0$ as,\\
$ \displaystyle \big| \hat R_{2}^{w}(x,y) \big| \leq \frac{1}{2!} \sum_{k \in \mathbb{Z}} \sum_{j \in \mathbb{Z}} \vert \chi(wx-k, wy-j) \vert \   \
\Big|  w^{2} \int_{\frac{k}{w}}^{\frac{k+1}{w}} \int_{\frac{j}{w}}^{\frac{j+1}{w}}  \Big \{  f_{xx} \big( \theta_{k,w}(u),\theta_{j,w}(v) \big)\\
\bigg( u - \frac{k}{w}\bigg)^{2}  + 2  f_{xy} \big( \theta_{k,w}(u),\theta_{j,w}(v) \big)
\bigg(u - \frac{k}{w}\bigg) \bigg( v - \frac{j}{w}\bigg)  \\+
   f_{yy} \big( \theta_{k,w}(u),\theta_{j,w}(v) \big) \bigg( v - \frac{j}{w}\bigg)^{2} \Big \} du dv  \Big| . $ \\

Thus, we have
$$ \big| \hat{R_{2}^{w}}(x,y)\big| \leq \sum_{k \in \mathbb{Z}} \sum_{j \in \mathbb{Z}} \big| \chi(wx-k, wy-j) \big| \Big[ \frac{1}{6 w^{2}} \mid f_{xx} \mid + \frac{1}{4 w^{2}} \mid f_{xy} \mid + \frac{1}{6 w^{2}} \mid f_{yy} \mid \Big] .$$ \\
Since, all the second order partial derivatives are bounded, i.e., $ \mid f_{xx}\mid \leq M_{1},$
$ \mid f_{xy} \mid \leq M_{2}, \  \mbox{and}  \mid f_{yy} \mid \leq M_{3},$  we have,
\begin{eqnarray} \label{e5}
\mid \hat{R_{2}^{w}}(x,y) \mid \leq \frac{7 M}{12 w^{2}} M_{(0,0)}^{0} \ ( \chi) \ < + \infty ,
\end{eqnarray}
where, $ M= \mbox{max} \{ M_{1},  M_{2},  M_{3} \}.$ \\

This is the desired resut.
\end{proof}

Now, we prove the inverse result by using the representation formula, derived in Theorem $3.$
\begin{theorem}
Let $ \chi$ be the kernel function satisfying the condition in (\ref{e3}) for every $ (u,v) \in \mathbb{R}^{2}$ with $r = 2.$ Let $f \in C^{2}(\mathbb{R}^{2})$ and suppose that;
$$  \Vert S_{w}f - f \Vert _{\infty} = o( w^{-1})$$
Then, $ f = g(y-x)$  for some arbitrary function $g.$
\end{theorem}

\begin{proof}{.}
By the representation formula of Theorem $3$, we have \\
$$ (S_{w}f)(x,y)
= (G_{w}f)(x,y)+ \frac{1}{2w} (G_{w}f_{x})(x,y) + \frac{1}{2w} (G_{w}f_{y})(x,y) +  \hat  R_{2}^{w}(x,y) $$
Now,
\begin{eqnarray} \label{e6}
w \Big[ S_{w}f)(x,y) - f(x,y) \Big] &=& w \Big[ (G_{w}f)(x,y) - f(x,y) \Big] + \frac{1}{2} (G_{w}f_{x})(x,y) + \frac{1}{2} G_{w}f_{y})(x,y) +  w \hat  R_{2}^{w}. \nonumber \\
&=& w \Big[ G_{w}f)(x,y) - f(x,y) \Big]+ I_{1} \ (say).
\end{eqnarray}
Now, we estimate $I_{1}$,
\begin{eqnarray*}
&I_{1} &= \frac{1}{2} (G_{w}f_{x})(x,y) + \frac{1}{2} (G_{w}f_{y})(x,y) +  w \hat  R_{2}^{w} \\
&&= \frac{1}{2} \big(G_{w}f_{x} - f_{x}) + \frac{1}{2} \big( G_{w}f_{y} - f_{y} ) + \frac{1}{2} \big(f_{x}+f_{y} \big)+ w \hat  R_{2}^{w}
\end{eqnarray*}
Using the convergence results of $ (G_{w}f_{x})$ and $ (G_{w}f_{y})$, we obtain,\\
$$ \lim_{ w \rightarrow \infty} \mid I_{1} \mid  \ \  \leq \lim_{ w \rightarrow \infty}  \Bigg | \frac{1}{2}(f_{x} + f_{y}) + w \hat  R_{2}^{w}  \Bigg | .$$

Since, $ \displaystyle{\lim_{ w \rightarrow \infty}} \big[ (S_{w}f)(x,y) - f(x,y) \big ] = 0 \ \ , \  \displaystyle{\lim_{ w \rightarrow \infty}} \big[ (G_{w}f)(x,y) - f(x,y) \big] = 0$ and  using (\ref{e5}), we have,
$\displaystyle \lim_{w \longrightarrow \infty} | w \ \hat R_{2}^{w}(x,y) | = 0 ,$ then the equation (\ref{e6}) gives,
  $$ f_{x} + f_{y} = 0. $$\\

The solution of the above first order partial differential equation is  given by,
$$  f = g(y-x)$$ for some arbitrary function $g.$
\end{proof}

\begin{remark}
If, in addition, $f$ satisfies the initial condition $ f(0,y)= f_{0}(y)= y ,$ then,
                 $$f(x,y) = (y-x) $$
i.e., if $f$ satisfies the initial condition  $ f(0,y)= f_{0}(y)= y ,$ then $f$ turns out to be linear.

\end{remark}

\begin{theorem} Let $\chi$ be the kernel satisfying the moment condition (\ref{e3}) with $ r \in N $ then,
$S_{w}$ maps algebraic polynomials of degree atmost $ r-1 $ into algebraic polynomials of the same degree.
\end{theorem}

\begin{proof}{.}
By using Theorem $3$ and the condition (\ref{e3}), we can easily obtain the desired result.
\end{proof}

\section{\textbf{Approximation in the Space of B$\ddot{o}$gel Continuous Functions}}
In recent years, the study of GBS operators of certain linear positive operators is an interesting topic in approximation theory and function theory. In this section, we shall give a generalization of the GBS operator for the $B$-continuous functions. For this, we shall introduce a GBS operator associated with the bivariate Kantorovich type sampling operators and study some of its smoothness properties.
Karl B$\ddot{o}$gel \cite{BOG1,BOG2} introduced the concepts of $B$-continuous and $B$-differentiable functions. In approximation theory, the well-known Korovkin theorem is developed for B-continuous functions by Badea et al. in \cite{BAD1,BAD2}. In \cite{DOB}, Dobrescu and Matei proved that any $B$-continuous function on a bounded interval can be approximated uniformly by boolean sum of bivariate Bernstein polynomials. The approximation properties of Bernstein-Stancu polynomials associated with GBS operators was considered in \cite{MIC}. Agrawal et al. \cite{PNA6} studied the degree of approximation for bivariate Lupa\c{s}-Durrmeyer type operators based on P\'{o}lya distribution with associated GBS operators. Recently, many researchers have made significant contributions on this topic. We refer the reader to some of the related papers (\cite{PNA3,PNA2,BAR,FAR,KAJ,POP}).

Let $X$ and $Y$ be compact intervals. A function $f:X\times X\rightarrow \mathbb{R}$ is called $B$-continuous (B$\ddot{o}$gel Continuous) at a point $(x_0,y_0)\in X\times Y $ if
$$\lim_{(x,y)\rightarrow (x_0,y_0)}\Delta_{(x,y)}f[x_0,y_o;x,y],$$ for any $(x_0,y_0)\in X\times Y,$ where $\Delta_{(x,y)}f[x_0,y_o;x,y]=f(x,y)-f(x,y_0)-f(x_0,y)+f(x_0,y_0).$ The function $f:X\times X\rightarrow \mathbb{R}$ is $B$-bounded if there exists $\Omega>0$ such that $\mid\Delta_{(x,y)}f[u,v;x,y]\mid \ \leq \Omega,$ for every $(x,y), (u,v)\in X\times Y .$
Throughout this paper, we denote $B_{b}(X\times Y)$ and $C_{b}(X\times Y)$ be the space of all $B$-bounded and $B$-continuous functions on $X\times Y$ respectively. Let $C_{b}(\mathbb{R}^2)$ be the space of all $B$-continuous functions defined on $\mathbb{R}^2.$

Motivated by the above work, we construct the bivariate Kantorovich type sampling operator $(S_{w}),$ associated with GBS operators, for any $f\in C_{b}(\mathbb{R}{^2})$ as,
$$ K_w^{\chi}(f(u,v);x,y) :=  S_{w}(f(u,y)+f(x,v)-f(u,v); x,y)$$
for all $ (x,y) \in \mathbb{R}^{2}.$ The GBS operator of bivariate Kantorovich type is defined as,\\

\noindent
$(K_w^{\chi})(f;x,y)$
\begin{eqnarray*} \label{K}
&=&\sum_{k=-\infty}^{\infty}\sum_{j=-\infty}^{\infty} \chi(wx-k, wy-j) \ w^{2}\int_{\frac{k}{w}}^{\frac{k+1}{w}} \int_{\frac{j}{w}}^{\frac{j+1}{w}}
(f(x,v)+f(u,y)-(u,v))\  du \ dv,
\end{eqnarray*}
where, $w\in \mathbb{R}^{+}$ and $(x,y)\in \mathbb{R}^2.$
We shall estimate the rate of convergence of the sequences of these operators to $f\in C_{b}(\mathbb{R}{^2})$ in terms of the modulus of continuity in B$\ddot{o}$gel sense.\\
Let $f\in C_{b}(\mathbb{R}{^2}).$ Then, the mixed modulus of smoothness of $f$ is given by,
$$\omega_{B}(f;\delta_1,\delta_2):=\sup\{\mid\Delta_{(x,y)}f[u,v;x,y]\mid: \mid x-u\mid <\delta_1, \mid y-v\mid<\delta_2\},$$ for all $(x,y), (u,v)\in\mathbb{R}^{2}$ and for any $(\delta_1,\delta_2)\in (0,\infty)\times (0,\infty)$ with $\omega_{B}:(0,\infty)\times (0,\infty)\rightarrow\mathbb{R}.$ The properties of $\omega_{B}$ can be found in (\cite{Badea C}).

\begin{theorem} \label{Theorem 6}
Let $ f \in C_{b}( \mathbb{R}^{2}).$ Then, the operator $(K_w^{\chi})$ satisfies the following estimate,
$$ \Big| (K_w^{\chi})(f;x,y) - f(x,y) \Big| \leq \Bigg( 1+ \frac{A}{\delta_{1}} + \frac{B}{\delta_{2}} + \frac{C}{\delta_{1} \delta_{2} }\Bigg). \omega_{B}(f; \delta_{1}, \delta_{2}) $$
where, $ A= \frac{1}{2w} \big[ M_{0,0}^0+ 2 M_{1,0}^1 \big],$ $B= \frac{1}{2w} \big[ M_{0,0}^0+ 2 M_{0,1}^1 \big]$ and, $C= \frac{1}{4w^{2}} \big[ M_{0,0}^0+ 2 M_{1,0}^1+ 2 M_{0,1}^1+ 4 M_{1,1}^2  \big].$
\end{theorem}

\begin{proof}{.}
Since, $ w_{B}(f;a \ \delta_{1},b \ \delta_{2}) \leq (1+a)(1+b)\  \omega_{mixed}(f; \delta_{1}, \delta_{2}) ,$
for $ a, b \ > 0,$
we have,
\begin{eqnarray} \nonumber \label{delta}
\mid\Delta_{(x,y)}f[u,v;x,y]\mid &\leq & \omega_{B} (f; |u-x|, |v-y|) \\
& \leq & \Bigg( 1+ \frac{|u-x|}{\delta_{1}} \Bigg) \ \Bigg(1+ \frac{|v-y|}{\delta_{2}} \Bigg) \omega_{B}(f; \delta_{1}, \delta_{2}) ,\\ \nonumber
\end{eqnarray}
for every $ (x,y) , (u,v) \in \mathbb{R}^{2}$ and for any $ \delta_{1}, \delta_{2} > 0.$ Now, applying the operator $(K_w^{\chi})$ on the definition of $\Delta_{(x,y)}f[u,v;x,y],$ we obtain,
$$ K_w^{\chi} (f;x,y)= f(x,y) - S_{w} (\Delta_{(x,y)}f[u,v;x,y])\ . $$

In veiw of (\ref{delta}), we have
\begin{eqnarray} \nonumber \label{8}
\Big| (K_w^{\chi})(f;x,y) - f(x,y) \Big| &\leq & \Big( 1 + \delta_{1}^{-1}. S_{w}(|u-x|;x,y)+ \delta_{2}^{-1}. S_{w}(|v-y|;x,y) + \\
&& \delta_{1}^{-1}.\delta_{2}^{-1} S_{w}(|u-x|.|v-y|;x,y) \Big). \omega_{B}(f;\delta_{1}, \delta_{2})
\end{eqnarray}

Using the definition of $(S_{w}),$ we obtain
\begin{eqnarray*}
S_{w} \big( |u-x| \big) &=& \sum_{k= -\infty}^{+\infty} \sum_{j= -\infty}^{+\infty} \chi(wx-k,wy-j) \ w^{2}\int_{\frac{k}{w}}^{\frac{k+1}{w}} \int_{\frac{j}{w}}^{\frac{j+1}{w}}  |u-x| du \ dv \\
&=& \frac{1}{2w}  \sum_{k= -\infty}^{+\infty} \sum_{j= -\infty}^{+\infty} \chi(wx-k,wy-j)[1+ 2(wx-k)] \\
&=& \frac{1}{2w} \big[ M_{0,0}^0+ 2 M_{1,0}^1 \big] :=A \ .
\end{eqnarray*}

Similarly, we get
\begin{eqnarray*}
S_{w} \big( |v-y| \big) &=& \sum_{k= -\infty}^{+\infty} \sum_{j= -\infty}^{+\infty} \chi(wx-k,wy-j) \ w^{2}\int_{\frac{k}{w}}^{\frac{k+1}{w}} \int_{\frac{j}{w}}^{\frac{j+1}{w}}   |v-y| du \ dv \\
&=& \frac{1}{2w} \big[ M_{0,0}^0+ 2 M_{0,1}^1 \big] :=B \  .
\end{eqnarray*}
and $ S_{w} \big( |u-x|.|v-y| \big) = \frac{1}{4w^{2}} \big[ M_{0,0}^0+ 2 M_{1,0}^1+ 2 M_{0,1}^1+ 4 M_{1,1}^2  \big] := C \  . $ On substituting the values of $A, B$ and $C$ in (\ref{8}), we get the desired result.
\end{proof}

Now, we define a B-differentiable (B$\ddot{o}$gel differentiable) function. A function $ f: X \times Y (\subset \mathbb{R}^{2}) \rightarrow \mathbb{R}$ is said to be B-differentiable at $ (x_{0},y_{0}),$ if the following limit exists finitely,
$$ \lim_{(x,y)\rightarrow (x_{0},y_{0})} \frac{\Delta_{(x,y)}f[x_{0},y_{o};x,y]}{(x-x_{0}) (y-y_{0})}.$$

The B-differential of $f$ at any point $(x_{0},y_{0})$ is denoted by $ D_{B}(f;x_{0},y_{0}).$ We denote the space of all B-differentiable functions defined on $ X \times Y $ as $ D_{b}(X \times Y).$

\begin{theorem} \label{Theorem 7}
Let $ f \in D_{b}(\mathbb{R}^{2})$ and $ D_{B}f \in B(\mathbb{R}^{2}).$ Then, for each $ (x,y) \in \mathbb{R}^{2},$ we have
\begin{eqnarray*}
\Big| (K_w^{\chi})(f;x,y) - f(x,y) \Big| &\leq &  D. \big(3 \|D_{B}f\|_{\infty}+ \omega_{B}(D_{B}f;\delta_{1},\delta_{2}) \big) + \\&&
\big[ E. \delta_{1}^{-1}+ F. \delta_{2}^{-1}+ G. \delta_{1}^{-1}\delta_{2}^{-1} \big]\omega_{B}(D_{B}f;\delta_{1},\delta_{2}) ,
\end{eqnarray*}
where, $D= \frac{1}{4w^{2}} \big[ M_{0,0}^0+ 2 M_{1,0}^1+ 2 M_{0,1}+ 4 M_{1,1}  \big] ,$ $E= \frac{1}{6w^{3}} \big[ M_{0,0}^1+ 3 M_{2,0}^2+ 3 M_{1,0}^1+ 2 M_{0,1}^1+ 6 M_{2,1}^3+ 6 M_{1,1}^2  \big],$ $F=  \frac{1}{6w^{3}} \big[ M_{0,0}^0+ 3 M_{0,2}^2+ 3 M_{0,1}^1+ 2 M_{1,0}^1+ 6 M_{1,2}^3+ 6 M_{1,1}^2  \big] $ and, $ G= \frac{1}{9 w^{4}} \big[ M_{0,0}^0+ 3 M_{2,0}^2+3 M_{0,2}^2+ 3 M_{0,1}^1+ 3 M_{1,0}^1+ 9 M_{2,2}^4+9 M_{1,2}^3+ 9 M_{2,1}^3+ 9 M_{1,1}^2  \big]  . $
\end{theorem}

\begin{proof}{.}
If $f \in D_{b}(\mathbb{R}^{2})$, then we have
$$ \Delta_{(x,y)}f[u,v;x,y] = (u-x) (v-y) D_{B} f(p,q) ,$$
where, $ p \in (x,u)$ and $ q \in (y,v).$ Using the definition of $\Delta_{(x,y)}f[u,v;x,y],$ we have
$$ D_{B} f(p,q) = \Delta_{(x,y)} D_{B} f(p,q)+   D_{B} f(p,y)+  D_{B} f(x,q) -  D_{B} f(x,y) .$$

Since, $ D_{B} f \in B(\mathbb{R}^{2}),$ using the above equality, we have \\

\noindent
$ \big |S_{w}( \Delta_{(x,y)}f[u,v;x,y];x,y)\big |$
\begin{eqnarray*}
&=& \big| S_{w} \big( (u-x)(v-y) D_{B} f(p,q); x,y \big| \\
& \leq & S_{w} \big( |u-x| |v-y| \big| \Delta_{(x,y)}f(p,q)\big|;x,y \big) +\\
&& S_{w} \Big( |u-x| |v-y| \big( | D_{B} f(p,y)| + | D_{B} f(x,q)| + | D_{B} f(x,y)| \big);x,y) \Big)\\
& \leq & S_{w} \Big(|u-x| |v-y|. \ \omega_{B} \big( D_{B} f; |p-x|,|q-y| \big);x,y \Big) + \\
&& 3 \| D_{B} f \|_{\infty} . S_{w} \big( |u-x| |v-y|;x,y \big).
\end{eqnarray*}

Now, Using the monotonocity of mixed modulus of smoothness $ \omega_{B}$, we can write\\

\noindent
$ \Big| (K_w^{\chi})(f;x,y) - f(x,y) \Big| $
\begin{eqnarray} \nonumber \label{9}
&\leq & 3 \|D_{B}f \|_{\infty} .S_{w} \big( |u-x| |v-y|; x,y \big) + \Big[S_{w} \big( |u-x| |v-y|; x,y \big)+\\ \nonumber
&& \delta_{1}^{-1}  S_{w} \big( |u-x|^{2} |v-y|; x,y \big) + \delta_{2}^{-1} S_{w} \big( |u-x| |v-y|^{2}; x,y \big)\\
&& \delta_{1}^{-1} \delta_{2}^{-1} S_{w} \big( (u-x)^{2} (v-y)^{2}; x,y \big) \Big] \omega_{B} \big( D_{B}f; \delta_{1},\delta_{2} \big).\\ \nonumber
\end{eqnarray}
By the definition of $(S_{w}),$ we get the following estimates,
\begin{eqnarray*}
S_{w} \big( |u-x| |v-y| \big) &=& \frac{1}{4w^{2}} \big[ M_{0,0}^0+ 2 M_{1,0}^1+ 2 M_{0,1}+ 4 M_{1,1}  \big] :=D \\
S_{w} \big( |u-x|^{2} |v-y| \big) &=& \frac{1}{6w^{3}} \big[ M_{0,0}^1+ 3 M_{2,0}^2+ 3 M_{1,0}^1+ 2 M_{0,1}^1+ 6 M_{2,1}^3+ 6 M_{1,1}^2  :=E \\
S_{w} \big( |u-x| |v-y|^{2} \big) &=&  \frac{1}{6w^{3}} \big[ M_{0,0}^0+ 3 M_{0,2}^2+ 3 M_{0,1}^1+ 2 M_{1,0}^1+ 6 M_{1,2}^3+ 6 M_{1,1}^2  \big] :=F\\
S_{w} \big( |u-x|^{2} |v-y|^{2} \big) &=& \frac{1}{9 w^{4}} \big[ M_{0,0}^0+ 3 M_{2,0}^2+3 M_{0,2}^2+ 3 M_{0,1}^1+ 3 M_{1,0}^1+ 9 M_{2,2}^4+ \\&&
9 M_{1,2}^3+ 9 M_{2,1}^3+ 9 M_{1,1}^2  \big] := G
\end{eqnarray*}
Using these estimates in (\ref{9}), we can easily get the result.

\end{proof}

The mixed K- functional (\cite{Badea,Cottin}) can be defined as,
$$ \hat{K}(f;t_{1},t_{2}, \eta) = \inf_{g_{1}, g_{2}, h} \big \{ \| f- g_{1}- g_{2}- h \|_{\infty} + t_{1} \| D_{B}^{2,0} g_{1} \|_{\infty} + t_{2} \| D_{B}^{0,2} g_{2} \|_{\infty} + \eta \| D_{B}^{2,2} h \|_{\infty} \big \}$$
where, $g_{1} \in C_{B}^{2,0},$ $g_{2} \in C_{B}^{0,2}$ and $ h \in C_{B}^{2,2}.$ Here, $ C_{B}^{p,q}$ denotes the space of the functions $ f \in C_{b}(\mathbb{R}^2)$ with continuous mixed partial derivatives $ D_{B}^{\xi,\eta}f,$ $ 0 \leq \xi \leq p,$ $0 \leq \eta \leq q.$ \\
The mixed partial derivatives are defined as,
$$D_{B}^{1,0}(f;x,x_{0}) = \lim_{x \rightarrow x_{0}} \frac{\Delta_{x}f \big([x,x_{0}];y_{0} \big)}{(x-x_{0})}$$
and,
$$D_{B}^{0,1}(f;y,y_{0}) = \lim_{y \rightarrow y_{0}} \frac{\Delta_{y}f \big(x_{0};[y_{0},y] \big)}{(y-y_{0})},$$
where, $ \Delta_{x}f \big([x,x_{0}];y_{0} \big)= f(x,y_{0})- f(x_{0},y_{0})$ and, $\Delta_{y}f \big(x_{0};[y_{0},y] \big)= f(x_{0},y)- f(x_{0},y_{0}).$ Similarly, we can define the higher order mixed partial derivative as in case of ordinary derivatives.

Now, we estimate the order of convergence in terms of mixed K- functional.

\begin{theorem} \label{Theorem 8}
Let $ (K_w^{\chi})$ be GBS operator of $ (S_{w})$. Then, for every $ f \in C_{b}(\mathbb{R}^{2}),$
$$ \big|  K_w^{\chi}(f;x,y) - f(x,y) \big|  \leq \hat{K}(f; H,J,L) .$$ \\
Here, $ H = \frac{1}{3 w^{2}} \big[ M_{0,0}+ 3 M_{2,0}+ 3 M_{1,0} \big],$ $ J= \frac{1}{3 w^{2}} \big[ M_{0,0}+ 3 M_{0,2}+ 3 M_{0,1} \big] $ and $L = \frac{1}{9 w^{4}} \big[ M_{0,0}+ 3 M_{2,0}+3 M_{0,2}+ 3 M_{0,1}+ 3 M_{1,0}+ 9 M_{2,2}+ 9 M_{1,2}+ 9 M_{2,1}+ 9 M_{1,1}  \big].$
\end{theorem}

\begin{proof}{.}
Since $ g_{1} \in C_{B}^{2,0},$ the Taylor's formula for $ g_{1}$ is given by
$$ g_{1}(u,v)= g_{1}(x,y) + (u-x). D_{B}^{1,0} g_{1}(x,y) + \int_{x}^{u} (u-t)D_{B}^{2,0} g_{1}(t,y) dt .$$

On applying the operator $ (K_w^{\chi})$ both the sides, we obtain
$$ \displaystyle K_w^{\chi}(g_{1},x,y) = g_{1}(x,y)+  K_w^{\chi} \Big( \int_{x}^{u} (u-t)D_{B}^{2,0} g_{1}(t,y) dt; x,y \Big).$$

Now,
\noindent
$ \big| K_w^{\chi}(g_{1};x,y) - g_{1}(x,y) \big| $
\begin{eqnarray*}
&=& \Bigg| S_{w} \Big(\int_{x}^{u} |u-t| \ \big| D_{B}^{2,0} g_{1}(t,y)- D_{B}^{2,0} g_{1}(t,v) dt \big| ; x,y \Big) \Bigg| \\
&\leq & \frac{\|D_{B}^{2,0} g_{1}\|_{\infty}}{2} .\ S_{w} \big( (u-x)^{2};x,y \big)
\end{eqnarray*}
Since,
$S_{w} \big( (u-x)^{2}\big) = \frac{1}{3 w^{2}} \big[ M_{0,0}^0+ 3 M_{2,0}^2+ 3 M_{1,0}^1 \big] := H,$ we have
$$\big| K_w^{\chi}(g_{1};x,y) - g_{1}(x,y) \big| \leq  \frac{\|D_{B}^{2,0} g_{1}\|_{\infty}}{2}  H  \ \ .$$

Similarly, we obatin
\begin{eqnarray*}
\big| K_w^{\chi}(g_{2};x,y) - g_{2}(x,y) \big| &\leq &  \frac{\|D_{B}^{0,2} g_{2}\|_{\infty} }{2}.\ S_{w} \big( (v-y)^{2};x,y \big)\\
&\leq & \frac{\|D_{B}^{0,2} g_{2}\|_{\infty} }{2} J \ ,
\end{eqnarray*}
where, $ J := S_{w} \big( (v-y)^{2};x,y \big) = \frac{1}{3 w^{2}} \big[ M_{0,0}^0+ 3 M_{0,2}^2+ 3 M_{0,1}^1 \big] .$ \\

For, $ h \in C_{B}^{2,2}(\mathbb{R}^{2}),$ we have
\begin{eqnarray*}
h(u,v)&=& h(x,y) + (u-x). D_{B}^{1,0} h(x,y) + (v-y). D_{B}^{0,1} h(x,y)+ \\
&& (u-x)(v-y). D_{B}^{1,1} h(x,y)+ \int_{x}^{u} (u-t)D_{B}^{2,0} h(t,y) dt+ \\
&& \int_{y}^{v} (v-s)D_{B}^{2,0} h(x,s) ds + \int_{x}^{u} (u-t)(v-y) D_{B}^{2,1} h(t,y) dt + \\
&& \int_{y}^{v} (u-t)(v-s) D_{B}^{2,0} h(x,s) ds+\int_{x}^{u} \int_{y}^{v} (u-t)(v-s) D_{B}^{2,0} h(t,s) dt ds.\\
\end{eqnarray*}
Using the definition of $(K_w^{\chi})$, we get
\begin{eqnarray*}
\big| K_w^{\chi}(h;x,y) - h(x,y) \big| &\leq & \Bigg| S_{w} \Big( \int_{x}^{u} \int_{y}^{v} (u-t)(v-s) D_{B}^{2,0} h(t,s) dt ds \Big) \Bigg|\\
& \leq & S_{w} \Big(\int_{x}^{u} \int_{y}^{v} |u-t| |v-s| \  \big|D_{B}^{2,0} h(t,s)\big| dt ds \Big)\\
&\leq & \frac{\|D_{B}^{2,0} h \|_{\infty}}{4} S_{w} \big( (u-x)^{2} (v-x)^{2} ;x,y \big) \\
& \leq &  \frac{\|D_{B}^{2,0} h \|_{\infty}}{4} L \ ,
\end{eqnarray*}
where, $L := S_{w} \big( (u-x)^{2} (v-x)^{2} ;x,y \big) =
\frac{1}{9 w^{4}} \big[ M_{0,0}^0+ 3 M_{2,0}^2+3 M_{0,2}^2+ 3 M_{0,1}^1+ 3 M_{1,0}^1+ 9 M_{2,2}^4+ 9 M_{1,2}^3+ 9 M_{2,1}^3+ 9 M_{1,1}^2  \big] .$\\

Thus, for $ f \in C_{b}(\mathbb{R}^{2}),$ we have\\

\noindent
$\big| K_w^{\chi}(f;x,y) - f(x,y) \big|$
\begin{eqnarray*}
& \leq & \big| (f- g_{1}-g_{2} -h)(x,y)\big| + \big| \big( g_{1} - K_w^{\chi} g_{1} \big) (x,y) \big| + \big| \big( g_{2} - K_w^{\chi} g_{2} \big) (x,y) \big| +\\
&& \big|  K_w^{\chi} (f- g_{1}-g_{2} -h);x,y \big| \\
& \leq & 2 \| f- g_{1}-g_{2} -h \|_{\infty} + H. \ \| D_{B}^{2,0} g_{1} \|_{\infty} + J. \ \| D_{B}^{2,0} g_{1} \|_{\infty} + L. \ \| D_{B}^{2,0} g_{1} \|_{\infty}
\end{eqnarray*}

Taking the infimum over all $ g_{1}  \in C_{B}^{2,0}, g_{2} \in C_{B}^{0,2}$ and $ h \in C_{B}^{2,2},$ we obtain the required result.

\end{proof}

\section{Examples of the kernel}
 Now, we give the kernel which will satisfy the assumptions (K1) to (K3) and condition (\ref{e3}) of section $2$. There are many examples of the univariate kernels satisfying the conditions (K1) to (K3) in one variable, given in \cite{PLB1,COS}, but some of the them fail to satisfy the moment condition (\ref{e3}), for eg., Fejer's kernel, see \cite{Dan}. To find the kernel satisfying the moment condition (\ref{e3}) also, Butzer and Stens  have given a result using well known B-spline kernels in \cite{Stens2}.\\
The central B-splines (univariate) of order $ n \in \mathbb{N} $ has the form ,
$$ M_{n}(x)= \frac{1}{(n-1)!} \sum_{j=0}^{n-1} (-1)^{j} \binom{n}{j} \Big( \frac{n}{2} +t-j \Big)_{+} ^{n-1} $$
where  $ t \in \mathbb{R}$ and $ x_{+}^{r} = max \{x^{r}, 0 \}.$\\

The Fourier transform of $ M_{n}$ is given as,
$$ \hat{M_{n}} = \Big(\frac{sin \  v/2}{v/2} \Big)^{n},$$
where $ n \in \mathbb{N} , v \in \mathbb{R}. $

\begin{theorem}{.}\cite{Stens2}{.}
For $ r \in \mathbb{N}$, $ r \geq 2,$ let $ \epsilon_{0} < \epsilon_{1} < ...< \epsilon_{r-1} $ be any given real numbers, and let $ a_{ \mu_{r}} $ , $ \mu=0,1,2,...,r-1,$ be the unique solutions of the linear system of equations;
$$ \sum_{\mu=0}^{r-1} a_{ \mu_{r}} \big( -i \epsilon_{\mu})^{j} = \Big( \frac{1}{M_{r}} \Big)^{j} (0),$$
for every $ j=0,1,2,...,r-1,$, where i denotes the imaginary unit. Then \\

$$ \chi_{r}(x)= \sum_{\mu=0}^{r-1} a_{ \mu r} M_{r} (t- \epsilon_{\mu})$$
is a polynomial spline of order $r$, satisfying the moment condition $(3)$ and having compact support contained in $ [ \epsilon_{0} - r/2, \epsilon_{r-1} + r/2].$\\
\end{theorem}

Here, for $ r=3 $ and $ \epsilon_{0} =2,$ we get,
\begin{equation} \label{e7}
\chi_{3}(x)= \frac{1}{8} \big( 47 M_{3}(x-2) - 62 M_{3}(x-2)+ 23 M_{3}(x-2) \big)
\end{equation}
where,
\begin{equation} \label{e8}
 M_{3}(x) =
     \begin{cases}
      \text{$ \frac{3}{4}- x^{2}$,} &\quad\text{if} \ \ { \mid x \mid \leq \frac{1}{2} ,} \\
     \text{$\frac{1}{2} \big( \frac{3}{2} - \mid x \mid  \big)^{2}$ ,} &\quad\text{if} \ \ {\frac{1}{2} < \mid x \mid \leq \frac{3}{2}, }\\
     \text{$0$,} &\quad\text{if} \ \ {\mid x \mid > \frac{3}{2} .}\\

   \end{cases}
\end{equation} for $x \in \mathbb{R}.$

 With the help of (\ref{e7}) and (\ref{e8}), we define a bivariate kernel as,
 $$ \bar{\chi}(x,y)= \chi_{3}(x). \chi_{3}(y) .$$
 As $ \chi_{3}(x)$ satisfies all the required conditions (see \cite{bv1}, \cite{Stens2}), so the kernel $ \bar{\chi}(x,y)$ will also satisfy all the assumptions of section $2.$ \\

The following result (see \cite{PLB2}) provides us another tool to construct some more examples of kernel satisfying the moment condition (\ref{e3}).

\begin{theorem}{.}{\cite{bv1}} Let $ \chi: \mathbb{R}^{2} \longrightarrow \mathbb{R}$ be a continuous and bounded function such that $ M_{r} (\chi) < + \infty $ for some $r \in \mathbb{N}_{0},$ $ p= (p_{1},p_{2}) \in\mathbb{N}_{0}^{2}$ with $|p|= p_{1}+p_{2} \leq r,$ the double series in the definition of $ M_{r} (\chi) $ is unifomly convergent on each compact subset of $ \mathbb{R}^{2}$ for each $p$ such that $ |p| =r ,$ then, the following two statements are equivalent for $ c \in \mathbb{R},$\\
(a) for every $ (x,y) \in \mathbb{R}^{2},$
$$ \sum_{k \in \mathbb{Z}} \sum_{j \in \mathbb{Z}} \chi(u-k,v-j) (u-k)^{p_{1}} (v-j)^{p_{2}} = c $$
(b) \begin{equation*}
\frac{\delta^{|p|} \hat{\chi}}{ \delta u ^{p_{1}} \delta v^{p_{2}}}(2k \pi,2j \pi) =
\begin{cases}
      \text{$(-1)^{|p|} \frac{c}{2 \pi}$ , } &\quad\text{if} \ \ { (k,j)=(0,0)} \\
     \text{0,} &\quad\text{if} \ \ {(k,j)\neq } (0,0)\\
   \end{cases}
\end{equation*}

\end{theorem}

Using the above result, we can see that the Bochner-Riesz kernel, defined as,
$$ \chi(x,y) = \frac{2^{\gamma} \ \Gamma ( \gamma +1)}{2 \pi} \big( \sqrt{ x^{2}+ y^{2}} \big) ^{-1- \gamma} J_{1+\gamma} \big( \sqrt{ x^{2}+ y^{2}} \big) ,$$
for $ \gamma > \frac{1}{2}$, where, $J_{\lambda}$ is the Bessel function of order $\lambda,$ also satisfies all the required conditions for $ \gamma > \frac{5}{2}$ (see \cite{bv1}).\\


\end{document}